\documentclass[11pt]{amsart}
\usepackage{amsfonts, amssymb, amsmath, amsthm, hyperref, color, float,enumerate}
\usepackage{tikz,pgffor,amsthm}
\theoremstyle{plain}
   \newtheorem{teo}{Theorem}
   
   \newtheorem{lema}[teo]{Lemma}
   \newtheorem{propo}[teo]{Proposition}
   
\theoremstyle{definition}
   
\theoremstyle{remark}
 \newtheorem{obs}{Remark}

\numberwithin{equation}{section}
\allowdisplaybreaks
\begin{document}

\title[Restricted weak type for one-sided operators]{Restricted weak type  inequalities for the one-sided Hardy-Littlewood maximal operator in higher dimensions}

\author[F. Berra]{Fabio Berra}
\address{CONICET and Departamento de Matem\'{a}tica (FIQ-UNL),  Santa Fe, Argentina.}
\email{fberra@santafe-conicet.gov.ar}


\thanks{The author was supported by CONICET and UNL}



\begin{abstract}
We give a quantitative characterization of the pairs of weights $(w,v)$ for which the dyadic version of the one-sided Hardy-Littlewood maximal operator satisfies a restricted weak $(p,p)$ type inequality, for $1\leq p<\infty$. More precisely, given any measurable set $E_0$ the estimate
\[w(\{x\in \mathbb{R}^n: M^{+,d}(\mathcal{X}_{E_0})(x)>t\})\leq \frac{C[(w,v)]_{A_p^{+,d}(\mathcal{R})}^p}{t^p}v(E_0)\]
holds if and only if the pair $(w,v)$ belongs to $A_p^{+,d}(\mathcal{R})$, that is
\[\frac{|E|}{|Q|}\leq [(w,v)]_{A_p^{+,d}(\mathcal{R})}\left(\frac{v(E)}{w(Q)}\right)^{1/p}\]
for every dyadic cube $Q$ and every measurable set $E\subset Q^+$. The proof follows some ideas appearing in \cite{O05}.

We also obtain a similar quantitative characterization for the non-dydadic case in $\mathbb{R}^2$ by following the main ideas in \cite{FMRO11}. 
\end{abstract}

\maketitle

\section{Introduction}
 In 1986 (\cite{Sawyer86}) E. Swayer started the theory of one-sided weights. Namely, he 
introduced the class of weights  $A_p^+$ and showed that this  class is necessary and sufficient for the weighted boundedness of the one-side Hardy-Littlewood maximal function. Some extensions and generalizations were given consequently in the articles \cite{MR93}, \cite{MRdlT93} and \cite{MROSdlT90}, among others. 

In \cite{O98} the author characterizes the functions $w$ for which the one-sided Hardy-Littlewood maximal operator 
\[M^+_v f(x)=\sup_{h>0}\frac{\int_x^{x+h}|f|v}{\int_x^{x+h}v}\]
verifies a restricted weak $(p,p)$ type on the real line, that is, a weak type inequality applied to the function $f=\mathcal{X}_E$, where $E$ is an arbitrary measurable set. More precisely, the inequality
\begin{equation*}
w\left(\left\{x\in \mathbb{R}: M^{+}_v(\mathcal{X}_E)(x)>t\right\}\right)\leq \frac{C}{t^p}v(E)
\end{equation*} 
holds if and only if $w\in A_p^+(\mathcal{R})(v\,dx)$. This set corresponds to the class of weights that satisfy a restricted $A_p^+$ condition with respect to the measure $d\mu=v(x)\,dx$ (see section below for details).

Although the theory in this setting was deeply developed and the main results were improved and generalized, most of the results were set on $\mathbb{R}$. 

In \cite{O05} Ombrosi characterized the pair of weights $(w,v)$ for those the inequality
\begin{equation}\label{eq: intro - tipo debil de M+ diadica}
w\left(\left\{x\in \mathbb{R}^n: M^{+,d}f(x)>t\right\}\right)\leq \frac{C}{t^p}\int_{\mathbb{R}^n}|f|^pv
\end{equation}
holds for every positive $t$, and where $1\leq p<\infty$. The operator $M^{+,d}$ is a dyadic version of $M^+$ defined on $\mathbb{R}^n$. A similar result was also obtained for $M^{-,d}$.

It is well-known that the operators $M^+f$ and $M^{+,d}f$ are pointwise equivalent on $\mathbb{R}$ (see \cite{MRdlT93}). However, this result is false in general in higher dimensions. This means that a non-dyadic version of \eqref{eq: intro - tipo debil de M+ diadica} cannot be obtained directly from the dyadic case, and the problem of find such an estimate remained open. 

In \cite{FMRO11} Forzani, Mart\'in-Reyes and Ombrosi proposed a way to generalize the operators $M^+$ and $M^-$ to higher dimensions and solved the problem discussed above on $\mathbb{R}^2$. The technique used, although newfangled and quite delicate, relied heavily upon the dimension. This means that the corresponding problem for $n\geq 3$ still remains open. 

 Related to strong estimates in dimension greater than one, some partial results were obtained in \cite{LO10}.  At this point we would also like to mention interesting applications of this theory to parabolic differential equations obtained by J. Kinnunen and O. Saari in \cite{KS16} and \cite{KS16-2}.

In this article we use some ideas of \cite{O05} and \cite{FMRO11} to give a characterization of the pairs of weights for which the one-sided Hardy-Littlewood maximal operator satisfies a restricted weak type inequality in higher dimensions.

Concretely, for the dyadic case we have the following result.

\begin{teo}\label{teo: teorema principal}
	Let $(w,v)$ be a pair of weights and $1\leq p<\infty$. Then the following statements are equivalent:
	\begin{enumerate}[\rm (a)]
		\item \label{item: teo: teorema principal - item a} The operator $M^{+,d}$ is of restricted weak $(p,p)$ type with respect to $(w,v)$, that is, there exists a positive constant $C$ such that the inequality
		\[w\left(\left\{x\in \mathbb{R}^n: M^{+,d}(\mathcal{X}_E)(x)>t\right\}\right)\leq \frac{C[(w,v)]_{A_p^{+,d}(\mathcal{R})}^p}{t^p}v(E)\]
		holds for every positive $t$ and every measurable set $E$;
		\item \label{item: teo: teorema principal - item b} $(w,v)$ belongs to $A_p^{+,d}(\mathcal{R})$.
	\end{enumerate}
\end{teo}

For the non-dyadic case we prove the next theorem.

\begin{teo}\label{teo: tipo debil (p,p) restringido para M^+ en R2}
	Let $(w,v)$ be a pair of nonnegative measurable functions defined in $\mathbb{R}^2$ and $1\leq~p<~\infty$. The following conditions are equivalent:
	\begin{enumerate}[\rm (a)]
		\item \label{item: teo: tipo debil (p,p) restringido para M^+ en R2 - item a} The operator $M^{+}$ is of restricted weak $(p,p)$ type with respect to $(w,v)$, that is, there exists a positive constant $C$ such that the inequality
		\[w\left(\left\{x\in \mathbb{R}^2: M^{+}(\mathcal{X}_E)(x)>t\right\}\right)\leq \frac{C[(w,v)]_{A_p^{+,d}(\mathcal{R})}^p}{t^p}v(E)\]
		holds for every positive $t$ and every measurable set $E$;
		\item \label{item: teo: tipo debil (p,p) restringido para M^+ en R2 - item b} $(w,v)$ belongs to $A_p^{+}(\mathcal{R})$.
	\end{enumerate}
\end{teo}

The article is organized as follows. In \S~\ref{section: preliminares} we give the preliminaries and definitions required for these main results. In \S~\ref{section: caso diadico} and \S~\ref{section: caso no diadico} we prove Theorem~\ref{teo: teorema principal} and Theorem~\ref{teo: tipo debil (p,p) restringido para M^+ en R2}, respectively.


\section{Preliminaries and basic definitions}\label{section: preliminares}
We shall deal with dyadic cubes with sides parallel to the coordinate axes. Given a dyadic cube $Q=\prod_{i=1}^n[a_i,b_i)$, we will denote with $Q^+=\prod_{i=1}^n[b_i,2b_i-a_i)$ and $Q^-=\prod_{i=1}^n[2a_i-b_i,a_i)$.

Given a positive number $s$, we denote $(Q)^{s,+}=\prod_{i=1}^n[a_i,a_i+sh)$, where $h=b_i-a_i$. Similarly, we denote $(Q)^{s,-}=\prod_{i=1}^n[b_i-sh,b_i)$.

For $x=(x_1,\dots,x_n)$ in $\mathbb{R}^n$ and $h>0$, we denote $Q_{x,h}=\prod_{i=1}^n[x_i,x_i+h)$ and $Q_{x,h^-}=\prod_{i=1}^n[x_i-h,x_i)$. The one-sided Hardy-Littlewood maximal operators are given by
\[M^{+}f(x)=\sup_{h>0}\frac{1}{|Q_{x,h}|}\int_{Q_{x,h}}|f(y)|\,dy, \quad\textrm{ and }\quad M^{-}f(x)=\sup_{h>0}\frac{1}{|Q_{x,h^-}|}\int_{Q_{x,h^-}}|f(y)|\,dy.\]
We shall consider the dyadic version of these operators, that is,
\[M^{+,d}f(x)=\sup_{Q\ni x}\frac{1}{|Q|}\int_{Q^+}|f(y)|\,dy,\quad\textrm{ and }\quad M^{-,d}f(x)=\sup_{Q\ni x}\frac{1}{|Q|}\int_{Q^-}|f(y)|\,dy.\]
where the supremum are taken over dyadic cubes. 

Given $1<p<\infty$, we say that a pair of weights $(w,v)$ belongs to $A_p^+$ if there exists a positive constant $C$ such that the inequallity
\[\left(\int_Q w\right)\left(\int_{Q^+} v^{1-p'}\right)^{p-1}\leq C|Q|^p\]
 holds for every cube $Q$ in $\mathbb{R}^n$. 
 
When $p=1$, we say that $(w,v)$ belongs to $A_1^+$ if there exist a positive constant $C$ that verifies
\[M^{-}w(x)\leq Cv(x),\]
for almost every $x$. The smallest constant for which these inequalities hold is denoted by $[(w,v)]_{A_p^+}$.
 
 Similarly, we say that $(w,v)$ belongs to $A_p^{+,d}$ if the inequalities above hold for every dyadic cube $Q$ and, when $p=1$, the involved operator is $M^{-,d}$. In this case, the corresponding smallest constant is denoted by $[(w,v)]_{A_p^{+,d}}$.

For $1\leq p<\infty$, we say that $(w,v)\in A_p^{+,d}(\mathcal{R})$ if there exists a positive constant $C$ such that the inequality
\begin{equation}\label{eq: condicion Ap^{+,d} restringido}
\frac{|E|}{|Q|}\leq C\left(\frac{v(E)}{w(Q)}\right)^{1/p}
\end{equation}
holds for every dyadic cube $Q$ and every measurable set $E\subset Q^+$. The smallest constant $C$ for which the inequality above holds will be denoted by $[(w,v)]_{A_p^{+,d}(\mathcal{R})}$.

We say that a pair of weights $(w,v)$ belongs to $A_p^+(\mathcal{R})$, $1\leq p<\infty$, if the inequality \eqref{eq: condicion Ap^{+,d} restringido} holds for every cube $Q$ and every measurable subset $E$ of $Q^+$.

\begin{obs}
By replacing $Q^+$ by $Q^-$ and $M^-$ by $M^+$ we can define the $A_p^-$ classes, for $1\leq p<\infty$. The  dyadic version of these classes, $A_p^{-,d}$, are defined by considering dyadic cubes on their definitions. The same occurs for $A_p^-(\mathcal{R})$ and $A_p^{-,d}(\mathcal{R})$. 

Throughout the paper we shall present the results for $M^+$, but the same arguments can be adapted to get the corresponding versions for $M^-$. 
\end{obs}

The novelty of considering restricted weak type inequalities relies on that although we take a particular function $f$, we consider a wider class of weights. This property is contained in the following proposition.

\begin{propo}\label{propo: propiedades de Ap^+ y Ap^+ restringido}
	 $A_p^{+}\subset A_p^{+}(\mathcal{R})$ for every $1<p<\infty$, and $A_1^+=A_1^+(\mathcal{R})$.
\end{propo}

\begin{proof}
	Let $1<p<\infty$ and assume that $(w,v)\in A_p^+$. Fix a cube $Q$ and a measurable subset $E$ of $Q^+$ with $|E|>0$. Then 
	\begin{align*}
	|E|&\leq \left(\int_E v\right)^{1/p}\left(\int_{Q^+}v^{1-p'}\right)^{1/p'}\\
	&\leq [(w,v)]_{A_p^+}^{1/p}\left(\frac{v(E)}{w(Q)}\right)^{1/p}|Q|,
	\end{align*}
	which implies that $(w,v)\in A_p^{+}(\mathcal{R})$ and $[(w,v)]_{A_p^{+}(\mathcal{R})}\leq [(w,v)]_{A_p^+}^{1/p}$. 
	
	On the other hand, set $p=1$ and assume that $(w,v)\in A_1^+$. Fix a cube $Q$ and a measurable set $E\subset Q^+$ with positive measure. Then, for every $x\in E$, we have that
	\[\frac{1}{|Q|}\int_{(Q^{+})^-}w\leq [(w,v)]_{A_1^+}v(x),\]
	which implies that
	\[\frac{w(Q)}{|Q|}\leq [(w,v)]_{A_1^+}\frac{v(E)}{|E|},\]
	and then $(w,v)\in A_1^+(\mathcal{R})$. Conversely, fix $x$ and $h>0$. Let $Q=Q_{x,h^{-}}$, $\lambda>\inf_{Q_{x,h}} v$ and $E=\{y\in Q_{x,h}: v(y)<\lambda\}$. Then we have that
	\[\frac{w(Q_{x,h^-})}{|Q_{x,h^-}|}\leq [(w,v)]_{A_1^+(\mathcal{R})}\lambda.\]
	By letting $\lambda\to \inf_{Q_{x,h}}v$ and then taking supremum over $h$ we get that
	\[M^-w(x)\leq [(w,v)]_{A_1^+(\mathcal{R})}v(x).\qedhere\]
\end{proof}

The following lemma states a useful property for weights on the $A_p^+(\mathcal{R})$ class.

\begin{lema}\label{lema: pesos truncados en Ap^+(R) estan en la clase}
	Let $1\leq p<\infty$, $(w,v)$ be a par of weights in $A_p^+(\mathcal{R})$ and $a,b$ two positive constants. Then
	\begin{enumerate}[\rm (a)]
		\item $(w_0,v_0)=(\max\{w,a\},\max\{v,b\})$ belongs to $A_p^+(\mathcal{R})$, provided $a\leq b$;
		\item $(w_1,v_1)=(\min\{w,a\},\max\{v,b\})$ belongs to $A_p^+(\mathcal{R})$. 
	\end{enumerate}
\end{lema}

\begin{proof}
	Let us first prove $(a)$. Fix a cube $Q$ and a measurable subset $E$ of $Q^+$. We have to show that there exists a positive constant $C$, independent of $Q$ and $E$, such that
	\[\frac{w_0(Q)}{v_0(E)}\leq C\left(\frac{|Q|}{|E|}\right)^p.\]
	We write
	\[w_0(Q)=\int_{Q\cap\{w\geq a\}}w_0+\int_{Q\cap\{w< a\}}w_0=w(Q\cap\{w\geq a\})+a|Q\cap\{w<a\}|,\]
	and therefore
	\[\frac{w_0(Q)}{v_0(E)}=\frac{w(Q\cap\{w\geq a\})}{v_0(E)}+\frac{a|Q\cap\{w<a\}|}{v_0(E)}=I+II.\]
	Now observe that
	\[I\leq \frac{w(Q)}{v(E)}\leq [(w,v)]_{A_p^+(\mathcal{R})}\left(\frac{|Q|}{|E|}\right)^p.\]
	On the other hand,
	\[II\leq \frac{a|Q|}{b|E|}\leq \left(\frac{|Q|}{|E|}\right)^p,\]
	since $a\leq b$ and $|Q|\geq |E|$. Therefore, $(w_0,v_0)\in A_p^+(\mathcal{R})$ and $[(w_0,v_0)]_{A_p^+(\mathcal{R})}\leq \max\left\{1,[(w,v)]_{A_p^+(\mathcal{R})}\right\}.$
	
	For the proof of $(b)$, observe that $w_1\leq w$ and $v_1\geq v$, so
	\[\frac{w_1(Q)}{v_1(E)}\leq \frac{w(Q)}{v(E)}\leq [(w,v)]_{A_p^+(\mathcal{R})}\left(\frac{|Q|}{|E|}\right)^{p},\]
	which shows that $(w_1,v_1)\in A_p^+(\mathcal{R})$ with $[(w_1,v_1)]_{A_p^+(\mathcal{R})}\leq [(w,v)]_{A_p^+(\mathcal{R})}$. 
\end{proof}

\section{Restricted weak $(p,p)$ type of $M^{+,d}$ in $\mathbb{R}^n$}\label{section: caso diadico}

We devote this section to prove~\ref{teo: teorema principal}. We start with a previous lemma which will be useful for this purpose. This result is an adaptation of Lemma 2.1 in \cite{O05}.

\begin{lema}\label{lema: estimacion w(Qj)}
	Let $1\leq p<\infty$, $(w,v)\in A_p^{+,d}(\mathcal{R})$ and $\mu>0$. Let $E$ be a measurable set such that $0<|E|<\infty$ and $\{Q_j\}_{j\in \Gamma_\mu}$ a disjoint family of dyadic cubes such that, for every $j\in \Gamma_\mu$, we have
	\begin{equation}\label{eq: lema: estimacion w(Qj) - eq1}
	\mu<\frac{|E\cap Q_j^+|}{|Q_j|}\leq 2\mu.
	\end{equation} 
	Then we have that
	\[\sum_{j\in \Gamma_\mu} w(Q_j)\leq \frac{C[(w,v)]_{A_p^{+,d}(\mathcal{R})}^p}{\mu^p}v\left(E\cap\left(\bigcup_{j\in \Gamma_\mu}Q_j^+\right)\right).\]
\end{lema}

\begin{proof}
	For $m\geq 0$, we define the sets
	\[i_m=\{j\in \Gamma_\mu: \textrm{ there exist exactly }m\textrm{ cubes }Q_s^+: Q_j^+\subsetneq Q_s^+ \textrm{ with }s\in\Gamma_\mu\}\]
	and also
	\[\sigma_m =\bigcup_{j\in i_m} Q_j^+.\]
	Also, we define $E_j^+=E\cap Q_j^+$ and $F_m=\bigcup_{j\in i_m} E_j^+$.
	
	Notice that $\Gamma_\mu=\bigcup_{m\geq 0} i_m$ and, if $j_1$ and $j_2$ belong to $i_m$ for some $m$, then $Q_{j_1}^+\cap Q_{j_2}^+=\emptyset$. This yields
	\[|F_m|=\sum_{j\in i_m} |E_j^+|.\]
	On the other hand, $\sigma_{m+1}\subset \sigma_m$ for every $m\geq 0$, so 
	\begin{equation}\label{eq: lema: estimacion w(Qj) - relacion conjuntos Fm}
	F_{m+1}\subset F_m\quad\textrm{ and }\quad|F_{m+1}|\leq |F_m|.
	\end{equation}
	For fixed $m_0$ and $j_0\in i_{m_0}$, if $Q_j^+\subsetneq Q_{j_0}^+$ then $j\in i_m$ with $m>m_0$ and $Q_j\subset Q_{j_0}^{2,+}$. Therefore,
	\[\bigcup_{m>m_0}\bigcup_{j\in i_m: Q_j^+\subsetneq Q_{j_0}^+} Q_j\subset (Q_{j_0})^{2,+}\] 
	and this implies that
	\[\sum_{m>m_0}\sum_{j\in i_m: Q_j^{+}\subsetneq Q_{j_0}^+}|Q_j|\leq \left|(Q_{j_0})^{2,+}\right|=2^n|Q_{j_0}|,\]
	since the cubes $Q_j$ are disjoint. Thus, by \eqref{eq: lema: estimacion w(Qj) - eq1} we get
	\begin{align*}
	\sum_{m>m_0}|F_m\cap Q_{j_0}^+|&=\sum_{m>m_0}\sum_{j\in i_m: Q_j^{+}\subsetneq Q_{j_0}^+}|E_j^+|\\
	&\leq 2\mu \sum_{m>m_0}\sum_{j\in i_m: Q_j^{+}\subsetneq Q_{j_0}^+}|Q_j|\\
	&\leq 2^{n+1}\mu|Q_{j_0}|\\
	&\leq 2^{n+1}|E_{j_0}^+|.
	\end{align*}
	This last estimate implies that
	\[\sum_{m=m_0+1}^{m_0+2^{n+2}}|F_m\cap Q_{j_0}^+|<2^{n+1}|E_{j_0}^+|\]
	and then there must be an index $m$, $m_0+1\leq m\leq m_0+2^{n+2}$ such that
	\[|F_m\cap Q_{j_0}^+|<\frac{|E_{j_0}^+|}{2}.\]
	By \eqref{eq: lema: estimacion w(Qj) - relacion conjuntos Fm} we get
	\[|F_{m_0+2^{n+2}}\cap Q_{j_0}^+|\leq |F_m\cap Q_{j_0}^+|<\frac{|E_{j_0}^+|}{2},\]
	and consequently
	\[\frac{|Q_{j_0}^+\cap F_{m_0+2^{n+2}}^c|}{|Q_{j_0}|}>\frac{1}{2}\frac{|E_{j_0}^+|}{|Q_{j_0}|}>\frac{\mu}{2}.\]
	Now, we can estimate
	\begin{align*}
	\sum_{j\in \Gamma_\mu} w(Q_j)&=\sum_{m=0}^\infty \sum_{j\in i_m}w(Q_j)\\
	&<\left(\frac{2}{\mu}\right)^p\sum_{m=0}^\infty\sum_{j\in i_m} w(Q_j)\left(\frac{|Q_j^+\cap F_{m+2^{n+2}}^c|}{|Q_j|}\right)^p\\
	&\leq \left(\frac{2}{\mu}\right)^p \sum_{m=0}^\infty\sum_{j\in i_m} [(w,v)]_{A_p^{+,d}(\mathcal{R})}^pw(Q_j)\frac{v(Q_j^+\cap F_{m+2^{n+2}}^c)}{w(Q_j)}\\
	&\leq \left(\frac{2}{\mu}\right)^p [(w,v)]_{A_p^{+,d}(\mathcal{R})}^p\sum_{m=0}^\infty\int_{\sigma_m-\sigma_{m+2^{n+2}}}\mathcal{X}_E v\\
	&=\left(\frac{2}{\mu}\right)^p [(w,v)]_{A_p^{+,d}(\mathcal{R})}^p\sum_{k=0}^{2^{n+2}-1}\sum_{m=0}^\infty \int_{\sigma_{2^{n+2}m+k}-\sigma_{2^{n+2}(m+1)+k}}\mathcal{X}_E v\\
	&\leq\left(\frac{2}{\mu}\right)^p [(w,v)]_{A_p^{+,d}(\mathcal{R})}^p\sum_{k=0}^{2^{n+2}-1}\int_{\sigma_k}\mathcal{X}_E v\\
	&\leq\frac{2^{n+p+2}[(w,v)]_{A_p^{+,d}(\mathcal{R})}^p}{\mu^p}v(E\cap \sigma_0).\qedhere
	\end{align*}
\end{proof}

\medskip

\begin{proof}[Proof of Theorem~\ref{teo: teorema principal}]
	We shall first prove that \eqref{item: teo: teorema principal - item a} implies \eqref{item: teo: teorema principal - item b}. Fix a dyadic cube $Q$ and a measurable subset $E$ of $Q^+$. Assume that $|E|>0$, since otherwise the condition follows inmediately. For every $x$ in $Q$ we have that
	\[M^{+,d}\mathcal{X}_E(x)\geq \frac{1}{|Q|}\int_{Q^+}\mathcal{X}_E=\frac{|E|}{|Q|},\]
	which implies that $Q\subset \{x: M^{+,d}\mathcal{X}_E(x)>|E|/(2|Q|)\}$. By using \eqref{item: teo: teorema principal - item a} we get
	\[w(Q)\leq C\left(\frac{|Q|}{|E|}\right)^pv(E),\]
	which shows that $(w,v)\in A_p^{+,d}(\mathcal{R})$.
	
	Now we prove that \eqref{item: teo: teorema principal - item b} implies \eqref{item: teo: teorema principal - item a}. Fix a measurable set $E$ and assume, without loss of generality, that $0<|E|<\infty$. Fix $t>0$ and let $\{Q_j\}_j$ the family of dyadic cubes that verify 
	\[w(\{x: M^{+,d}\mathcal{X}_E(x)>t\})=\bigcup_j Q_j,\]
	where $|E\cap Q_j^+|/|Q_j|>t$ for every $j$. We shall consider a partition of this family of cubes. Given $k\geq 0$, we set
	\[C_k=\left\{j: 2^k t<\frac{|E\cap Q_j^+|}{|Q_j|}\leq 2^{k+1}t\right\}\]
	and apply Lemma~\ref{lema: estimacion w(Qj)} to the family $C_k$ with $\mu=2^kt$, for every $k$. Therefore,
	\[\sum_{j\in C_k}w(Q_j)\leq \frac{C[(w,v)]_{A_p^{+,d}(\mathcal{R})}^p}{(2^k t)^p}v\left(\bigcup_{j\in C_k}E_j^+\right).\] 
	This yields
	\begin{align*}
	w(\{x: M^{+,d}\mathcal{X}_E(x)>t\})&=\sum_jw(Q_j)\\
	&=\sum_{k=0}^\infty\sum_{j\in C_k}w(Q_j)\\
	&\leq \sum_{k=0}^\infty \frac{C[(w,v)]_{A_p^{+,d}(\mathcal{R})}^p}{(2^kt)^p}v(E)\\
	&=\frac{C[(w,v)]_{A_p^{+,d}(\mathcal{R})}^p}{t^p}v(E),
	\end{align*}
	which completes the proof.\qedhere
\end{proof}

\medskip

\section{Restricted weak $(p,p)$ type of $M^+$ in $\mathbb{R}^2$}\label{section: caso no diadico}

We devote this section to the proof of Theorem~\ref{teo: tipo debil (p,p) restringido para M^+ en R2}. Along this section we shall assume that the space where we work is $\mathbb{R}^2$. We begin by introducing some specifics in this setting.

We say that a square $Q$ has dyadic size if $\ell(Q)=2^k$, for some integer $k$. $\ell(Q)$ denotes the length of the sides of $Q$. Given a square $Q$, $\alpha Q$ will denote the square with the same center as $Q$ and sides of length $\alpha \ell(Q)$.

For $h>0$ and $Q=[a,a+h]\times [b,b+h]$, we set $\tilde{Q}$ the dilation of $Q$ to the right and to the bottom in $\ell(Q)/2$. That is, $\tilde{Q}=[a,a+3/2h]\times[b-h/2,b+h]$.

\medskip

\begin{figure}[h!]\label{fig: Q y Q tilde}
\begin{center}
\begin{tikzpicture}
\draw [line width=1.25] (0,0)--(3,0)--(3,3)--(0,3)--cycle;
\draw [line width=1.25] (2,3)--(2,1)--(0,1);
\node at (1,2) {$Q$};
\node at (2,0.5) {$\tilde{Q}$};
\node [left] at (0,2) {$h$};
\node [left] at (0,0.5) {$\frac{h}{2}$};
\node [above] at (1,3) {$h$};
\node [above] at (2.5,3) {$\frac{h}{2}$};
\end{tikzpicture}
\end{center}
\caption{The cubes $Q$ and $\tilde{Q}$.}
\end{figure}
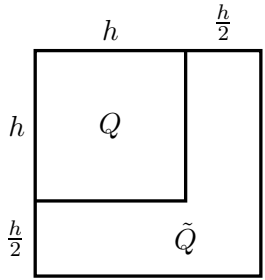

\medskip

Given $x=(x_1,x_2)\in \mathbb{R}^2$ and $h>0$ recall that $Q_{x,h}=[x_1,x_1+h]\times [x_2,x_2+h]$. We shall consider the following partition of a cube $Q_{x,h}$:
\[Q_{x,h}=Q_{x,h/2}\cup Q_{x,h}^1\cup Q_{x,h}^2\cup Q_{x,h}^3,\]
where
\begin{align*}
Q_{x,h}^1&=[x_1+h/2,x_1+h]\times[x_2+h/2,x_2+h],\\
Q_{x,h}^2&=[x_1+h/2,x_1+h]\times[x_2,x_2+h/2],\\
Q_{x,h}^3&=[x_1,x_1+h/2]\times[x_2+h/2,x_2+h].
\end{align*}

\begin{figure}[h!]
	\begin{center}
		\begin{tikzpicture}
		\draw [line width=1.25] (0,0)--(4,0)--(4,4)--(0,4)--cycle;
		\draw (2,0)--(2,4);
		\draw (0,2)--(4,2);
		\node at (1,3) {$Q_{x,h}^3$};
		\node at (3,3) {$Q_{x,h}^1$};
		\node at (3,1) {$Q_{x,h}^2$};
		\node at (1,1) {$Q_{x,\tfrac{h}{2}}$};
		\end{tikzpicture}
	\end{center}
	\caption{Subsquares of $Q_{x,h}$.}
	\label{fig: subcuadrados}
\end{figure}
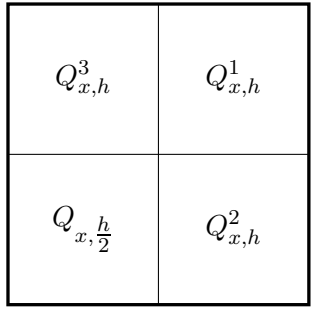

The proof of Theorem~\ref{teo: tipo debil (p,p) restringido para M^+ en R2} relies on the following covering lemma, that is a consequence of Lemma~3.1 stated and proved in \cite{FMRO11}, when we take $f=\mathcal{X}_E$. This result contains a covering argument that is related to the subcube $Q_{x,h}^2$. For the main proof we will require the corresponding versions for $Q_{x,h}^1$ and $Q_{x,h}^3$, that can be achieved by following similar ideas.

\begin{lema}\label{lema: cubrimiento para i=2}
	Let $t>0$ and $E$ a measurable set such that $0<|E|<\infty$. Let $A=\{x_j\}_{j=1}^N$ be a finite  set of points in $\mathbb{R}^2$. Suppose that, for every $1\leq j\leq N$, we have a square of dyadic size $Q_j$, with $x_j$ as its upper right corner and that satisfies
	\[\frac{t}{4}<\frac{\left|E\cap Q_j^{+2}\right|}{|Q_j|}.\]
	Then there exists a set $\Gamma\subset\{1,\dots,N\}$ such that
	\begin{equation}
	A\subset \bigcup_{i\in \Gamma} \tilde{Q}_i
	\end{equation}
	and
	\begin{equation}
	\frac{t}{4}<\frac{\left|E\cap \tilde{Q}_j^+\right|}{|Q_j|}.
	\end{equation}
	Moreover, for every $i,j\in \Gamma$ with $i\neq j$ we have $\tilde{Q}_i\not\subseteq \tilde{Q}_j$ and the squares $\tilde{Q}_i$, $i\in \Gamma$, of the same size are almost disjoint, that is, there exists a positive constant $C$ such that for every $l$
	\[\sum_{i\in \Gamma, \ell(Q_i)=l}\mathcal{X}_{\tilde{Q}_i}(x)\leq C.\]
	This implies that the squares $(\tilde{Q}_i)^+$ are almost disjoint too. Further, if
\[\frac{|E\cap (\tilde{Q}_j)^+|}{|Q_j|}\leq 8t\]
then there exists a family of sets $\{F_j\}_{j\in \Gamma}$ with $F_j\subset (\tilde {Q}_j)^+$ such that
\[\frac{t}{8}<\frac{|E\cap F_j|}{|Q_j|}\]
and they are almost disjoint, that is, there exists a positive constant $C$ (independent of everything) such that
\[\sum_{j\in \Gamma} \mathcal{X}_{F_j}(x)\leq C. \]
\end{lema}

\begin{proof}[Proof of Theorem~\ref{teo: tipo debil (p,p) restringido para M^+ en R2}]
	The fact that \eqref{item: teo: tipo debil (p,p) restringido para M^+ en R2 - item a} implies \eqref{item: teo: tipo debil (p,p) restringido para M^+ en R2 - item b} can be achieved in a similar way to Theorem~\ref{teo: teorema principal}. Let us prove then that \eqref{item: teo: tipo debil (p,p) restringido para M^+ en R2 - item b} implies \eqref{item: teo: tipo debil (p,p) restringido para M^+ en R2 - item a}.
	The operator $M^+$ is pointwise equivalent to the operator
	\[\mathcal{M}^+f(x)=\sup_{k\in\mathbb{Z}}\frac{1}{|Q_{x,2^k}|}\int_{Q_{x,2^k}}|f|,\]
	that is, the one-sided maximal operator defined over squares of dyadic size.
	We shall consider, for $i=1,2$ and 3 the operators
	\[\mathcal{M}^{+i}f(x)=\sup_{k\in\mathbb{Z}}\frac{1}{|Q_{x,2^k}^i|}\int_{Q_{x,2^k}^i}|f|,\]
	where the cubes $Q_{x,2^k}^i$ are depicted in Figure~\ref{fig: subcuadrados}.
	
	Let us fix a measurable set $E$ with $0<|E|<\infty$. Let $(w,v)$ be a pair of weights in $A_p^+(\mathcal{R})$. We shall prove that
	\[w\left(\left\{x\in \mathbb{R}^2: \mathcal{M}^+(\mathcal{X}_E)(x)>t\right\}\right)\leq \frac{C}{t^p}v(E),\]
	for every $t>0$. It will be enough to show that
	\[w(\{x\in \mathbb{R}^2: t<\mathcal{M}^+(\mathcal{X}_E)(x)\leq 2t\})\leq \frac{C}{t^p}v(E),\]
	and this also reduces to prove that
	\begin{equation}\label{eq: teo: tipo debil (p,p) restringido para M^+ en R2 - eq1}
	w(\{x\in \mathbb{R}^2: t<\mathcal{M}^{+i}(\mathcal{X}_E)(x), \quad \mathcal{M}^+(\mathcal{X}_E)(x)\leq 2t\})\leq \frac{C}{t^p}v(E),
	\end{equation}
	for $i=1,2,3$. We show the proof for $i=2$, being similar for the other indices.
	
	Given a positive number $\xi$ we consider the truncated maximal operator defined by
	\[M_\xi^{+2}(\mathcal{X}_E)(x)=\sup_{h=2^k>\xi, k\in\mathbb{Z}}\frac{4|E\cap Q_{x,h}^2|}{h^2}.\]
	Observe that $M_\xi^{+2}(\mathcal{X}_E)\uparrow \mathcal{M}^{+2}(\mathcal{X}_E)$ when $\xi\to 0^+$. Therefore, it will be enough to prove that
	\begin{equation}\label{eq: teo: tipo debil (p,p) restringido para M^+ en R2 - eq2}
	w\left(\left\{x\in \mathbb{R}^2: t<\mathcal{M}_\xi^{+2}(\mathcal{X}_E)(x), \quad \mathcal{M}^+(\mathcal{X}_E)(x)\leq 2t\right\}\right)\leq \frac{C}{t^p}v(E),
	\end{equation}
	for every $t>0$ and with $C$ independent of $\xi, E$ and $t$.
	
	By virtue of Lemma~\ref{lema: pesos truncados en Ap^+(R) estan en la clase} we can assume that $w\in L^1_{\textit{loc}}$ and also that there exists a positive constant $\gamma$ such that $0<\gamma\leq w(x)$, for every $x\in \mathbb{R}^2$.
	
	Let $\Omega_t=\left\{x\in \mathbb{R}^2: t<\mathcal{M}_\xi^{+2}(\mathcal{X}_E)(x), \quad \mathcal{M}^+(\mathcal{X}_E)(x)\leq 2t\right\}$. The measure $d\mu(x)=w(x)\,dx$ is finite over compact sets since we are assuming $w\in L^1_{\textit{loc}}$. Therefore, inequality \eqref{eq: teo: tipo debil (p,p) restringido para M^+ en R2 - eq2} follows if we prove that
	\[w(K)\leq \frac{C}{t^p}v(E),\]
	for every compact set $K\subset \Omega_t$ and with $C$ independent of $K$.
	
	Fix a compact set $K\subset \Omega_t$. For every $x=(x_1,x_2)\in K$ there exists a square $Q_x=[x_1-~\ell,x_1]\times[x_2-\ell,x_2]$ with $\xi\leq \ell$, $\ell=2^k$ for some $k\in \mathbb{Z}$ and
	\[\frac{t}{4}<\frac{|E\cap Q_x^{+2}|}{|Q_x|}.\]
	Let $Q_{x,2\ell}=[x_1,x_1+2\ell]\times[x_2,x_2+2\ell]$. We have that $(\tilde{Q}_x)^{+2}\subset Q_{x,2\ell}$ (see Figure~\ref{fig: contencion de cuadrados}) and thus
	
	\begin{align*}
	\frac{|E\cap (\tilde{Q}_x)^{+2}|}{|Q_x|}&\leq \frac{|E\cap Q_{x,2\ell}|}{|Q_x|}\\
	&=\frac{4|E\cap Q_{x,2\ell}|}{|Q_{x,2\ell}|}\\
	&\leq 4\mathcal{M}^+(\mathcal{X}_E)(x)\leq 8t.
	\end{align*}
	
	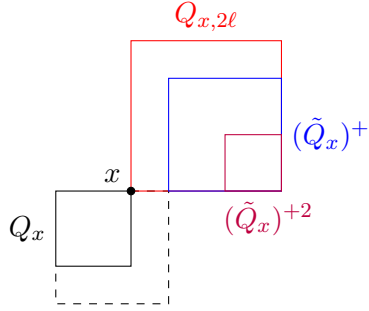
\begin{figure}[h!]
		\begin{center}
			\begin{tikzpicture}
			\draw (0,0)--(0,-1)--(-1,-1)--(-1,0)--cycle;
			\draw[dashed] (0,0)--(0.5,0)--(0.5,-1.5)--(-1,-1.5)--(-1,-1);
			\draw[color=red] (0,0)--(2,0)--(2,2)--(0,2)--cycle;
			\draw[color=blue] (0.5,0)--(0.5,1.5)--(2,1.5)--(2,0)--cycle;
			\node[left] at (0,0.2) {$x$};
			\draw [fill=black] (0,0)circle(0.05cm);
			\node[left] at (-1,-0.5) {$Q_x$};
			\node[right, color=blue] at (2,0.75) {$(\tilde{Q}_x)^+$};
			\node[above, color=red] at (1,2) {$Q_{x,2\ell}$};
			\draw[color=purple] (1.25,0)--(1.25,0.75)--(2,0.75)--(2,0)--cycle;
			\node[below,color=purple] at (1.825,0) {$(\tilde{Q}_x)^{+2}$};
			\end{tikzpicture}
		\end{center}
	\caption{$(\tilde{Q}_{x})^{+2}\subset Q_{x,2\ell}$.}
	\label{fig: contencion de cuadrados}
	\end{figure} 

Therefore, we have that for every $x\in K$ there exists a square $Q_x=[x_1-\ell,x_1]\times[x_2-\ell,x_2]$ such that $\xi\leq \ell$,
\[\frac{t}{4}<\frac{\left|E\cap Q_x^{+2}\right|}{|Q_x|}\]
and
\[\frac{\left|E\cap (\tilde{Q}_x)^{+2}\right|}{|Q_x|}\leq 8t.\]

We have also that there exists a positive constant $M$, depending on $t$ and $E$, such that $\ell\leq M$ since
\[|Q_x|\leq \frac{4\left|E\cap Q_x^{+2}\right|}{t}\leq\frac{4|E|}{t}<\infty.\]
This implies that there exists a square $R$ such that $\bigcup_{x\in K} \tilde{Q}_x\subset R$. We shall consider the square $2R$. Since $w$ is integrable in $2R$, there exists $0<\varepsilon<1$ such that if $Q\subset R$ is a square, then
\[w((1+\varepsilon)Q\backslash Q)\leq \gamma \xi^2.\]
If $Q\subset R$ verifies $\ell(Q)\geq \xi$, then
\[w((1+\varepsilon)Q\backslash Q)\leq \gamma\xi^2\leq \gamma|Q|\leq w(Q).\]
This yields
\[w((1+\varepsilon)Q)\leq 2w(Q),\]
for every $Q\subset R$ with $\ell(Q)\geq \xi$. Particularly,
\[w((1+\varepsilon)\tilde{Q}_x)\leq 2w(\tilde{Q}_x), \quad \textrm{ for every }x\in K.\]
Let $B_x(r)$ be the ball of radius $r$ centered at $x$. We have that $K\subset \bigcup_{x\in K}B_x\left(\frac{\xi\varepsilon}{2}\right)$, and then there exist $x_1,x_2,\dots,x_s\in K$ such that $K\subset\bigcup_{j=1}^s B_{x_j}\left(\frac{\xi\varepsilon}{2}\right)$, since $K$ is compact.

We apply now Lemma~\ref{lema: cubrimiento para i=2} to the set $A=\{x_j\}_{j=1}^s$ and the squares $\{Q_j\}_{j=1}^s$ associated to the points $x_j$. Then, there exists a set $\Gamma\subset\{1,\dots,s\}$ that verifies $A\subset \bigcup_{i\in \Gamma} \tilde{Q}_{x_i}$ and there also exist $\{F_{x_i}: i\in \Gamma\}$, $F_{x_i}\subset (\tilde{Q}_{x_i})^+$,
\begin{equation}\label{eq: teo: tipo debil (p,p) restringido para M^+ en R2 - eq3}
\frac{t}{8}<\frac{|E\cap F_{x_i}|}{|Q_{x_i}|}
\end{equation}
and
\[\sum_{i\in \Gamma}\mathcal{X}_{F_{x_i}}(x)\leq C.\]
Observe that if $x_j\in A$, there exists $i\in \Gamma$ such that $x_j\in \tilde{Q}_{x_i}$. Then $B_{x_j}\left(\frac{\xi\varepsilon}{2}\right)\subset (1+\varepsilon)\tilde{Q}_{x_i}$. In fact, this is straightforward if we assume $0<\xi<1$. Consequently, we have that
\[K\subset \bigcup_{j=1}^s B_{x_j}\left(\frac{\xi\varepsilon}{2}\right)\subset \bigcup_{i\in \Gamma}(1+\varepsilon)\tilde{Q}_{x_i},\]
which implies that
\[w(K)\leq \sum_{i\in \Gamma} w((1+\varepsilon)\tilde{Q}_{x_i})\leq 2\sum_{i\in \Gamma} w(\tilde{Q}_{x_i}).\]
Thus, by using \eqref{eq: teo: tipo debil (p,p) restringido para M^+ en R2 - eq3} and the $A_p^+(\mathcal{R})$ condition of $(w,v)$, we obtain
\begin{align*}
w(K)&\leq 2\sum_{i\in \Gamma}w(\tilde{Q}_{x_i})\\
&\leq \frac{C}{t^p}\sum_{i\in \Gamma} w(\tilde{Q}_{x_i})\left(\frac{|E\cap F_{x_i}|}{|Q_{x_i}|}\right)^p\\
&=\frac{C}{t^p}\sum_{i\in \Gamma}w(\tilde{Q}_{x_i})\left(\frac{|(\tilde{Q}_{x_i})^+|}{|Q_{x_i}|}\right)^p\left(\frac{|E\cap F_{x_i}|}{|(\tilde{Q}_{x_i})^+|}\right)^p\\
&\leq \frac{C}{t^p} [(w,v)]_{A_p^+(\mathcal{R})}^p \sum_{i\in \Gamma} v(E\cap F_{x_i})\\
&\leq \frac{C}{t^p} [(w,v)]_{A_p^+(\mathcal{R})}^pv\left(E\cap\left(\bigcup_{i\in \Gamma} F_{x_i}\right)\right)\\
&\leq \frac{C}{t^p} [(w,v)]_{A_p^+(\mathcal{R})}^pv(E).\qedhere
\end{align*}
	
\end{proof}
 
\medskip

\section{Acknowledgements}

I would like to specially thank Ph. D. Sheldy Ombrosi for suggesting me these one-sided problems, as well as reading the manuscript and giving me useful advices for redaction and bibliography.


\def\cprime{$'$}
\providecommand{\bysame}{\leavevmode\hbox to3em{\hrulefill}\thinspace}
\providecommand{\MR}{\relax\ifhmode\unskip\space\fi MR }
\providecommand{\MRhref}[2]{%
	\href{http://www.ams.org/mathscinet-getitem?mr=#1}{#2}
}
\providecommand{\href}[2]{#2}

\end{document}